\documentclass[12pt]{amsart}%
\usepackage{amsmath}
\usepackage{amsfonts}
\usepackage{amssymb}
\usepackage{graphicx}
\usepackage{MnSymbol}
\usepackage{comment}
\usepackage[margin= 2 in]{geometry}%
\providecommand{\U}[1]{\protect\rule{.1in}{.1in}}
\newtheorem{theorem}{Theorem}

\newtheorem{corollary}[theorem]{Corollary}

\newtheorem{definition}[theorem]{Definition}
\newtheorem{example}[theorem]{Example}

\newtheorem{lemma}[theorem]{Lemma}

\newtheorem{proposition}[theorem]{Proposition}
\newtheorem{remark}[theorem]{Remark}

\allowdisplaybreaks[1]

\theoremstyle{definition}
\numberwithin{equation}{section}

\newcommand{\resumename}{R\'esum\'e}

\begin{document}
\date{\today}
\title[Dihedral Group Frames]{Dihedral Group Frames which Are Maximally Robust to Erasures}
\author[V. Oussa]{Vignon Oussa}
\address{(4): Dept.\ of Mathematics\\
Bridgewater State University\\
Bridgewater, MA 02324 U.S.A.\\}
\email{Vignon.Oussa@bridgew.edu}
\maketitle

\begin{abstract}
Let $n$ be a natural number larger than two. Let $D_{2n}=\langle r,s :
r^{n}=s^{2}=e, srs=r^{n-1} \rangle$ be the Dihedral group, and $\kappa $ an $n$-dimensional unitary representation of $D_{2n}$ acting in $\mathbb{C}^n$ as follows. $(\kappa (r)v)(j)=v((j-1)\mod n)$ and $(\kappa(s)v)(j)=v((n-j)\mod n)$ for $v\in\mathbb{C}^n.$ For any representation which is unitarily equivalent to $\kappa,$ we prove that when $n$ is prime there exists a
Zariski open subset $E$ of $\mathbb{C}^{n}$ such that for any vector $v\in E,$
any subset of cardinality $n$ of the orbit of $v$ under the action of this representation is a basis for
$\mathbb{C}^{n}.$ However, when $n$ is even there is no vector in
$\mathbb{C}^{n}$ which satisfies this property. As a result, we derive that if
$n$ is prime, for almost every (with respect to Lebesgue measure) vector $v$ in $\mathbb{C}^{n}$ the $\Gamma
$-orbit of $v$ is a frame which is maximally robust to erasures. We also
consider the case where $\tau$ is equivalent to an irreducible unitary representation of the
Dihedral group acting in a vector space $\mathbf{H}_{\tau}\in\left\{\mathbb{C},\mathbb{C}^2\right\}$ and we provide
conditions under which it is possible to find a vector $v\in\mathbf{H}_{\tau}$
such that $\tau\left(  \Gamma\right)  v$ has the Haar property.

\end{abstract}

\section{Introduction}

Let $\mathfrak{F}$ be a set of $m\geq n$ vectors in an $n$-dimensional vector
space $\mathbb{K}^{n}$ over a field $\mathbb{K}\in\left\{  \mathbb{R}%
,\mathbb{C}\right\}  .$ We say that $\mathfrak{F}$ has the \textbf{Haar
property} if any subset of $\mathfrak{F}$ of cardinality $n$ is a basis for
$\mathbb{K}^{n}$. Let
\[
T=\left(
\begin{array}
[c]{ccccc}%
0 & 0 & \cdots & 0 & 1\\
1 & 0 & \cdots & 0 & 0\\
0 & 1 & \ddots & \vdots & \vdots\\
\vdots & \ddots & \ddots & 0 & 0\\
0 & \cdots & 0 & 1 & 0
\end{array}
\right)  \text{ and }M=\left(
\begin{array}
[c]{ccccc}%
1 &  &  &  & \\
& \exp\left(  \frac{2\pi i}{n}\right)   &  &  & \\
&  & \exp\left(  \frac{4\pi i}{n}\right)   &  & \\
&  &  & \ddots & \\
&  &  &  & \exp\left(  \frac{2\pi i\left(  n-1\right)  }{n}\right)
\end{array}
\right)
\]
be two invertible matrices with complex entries. The group generated by these
matrices is isomorphic to the finite Heisenberg group
\[
\mathrm{Heis}\left(  n\right)  =\left\{  \left(
\begin{array}
[c]{ccc}%
1 & m & k\\
0 & 1 & l\\
0 & 0 & 1
\end{array}
\right)  :\left(  k,m,l\right)  \in\mathbb{Z}_{n}\times\mathbb{Z}_{n}%
\times\mathbb{Z}_{n}\right\}
\]
which is a nilpotent group. It is well-known that (see \cite{Pfander}) if $n$
is prime then there exists a Zariski open set $E$ of $\mathbb{C}^{n}$ such
that for every vector $v\in E,$ the set $\left\{  M^{l}T^{k}v:1\leq l,k\leq
n\right\}  $ has the Haar property. This special property has some important
application in the theory of frames. We recall that a frame\textbf{ (}see
\cite{Casazza}\textbf{)} in a Hilbert space is a sequence of vectors $\left(
u_{k}\right)  _{k\in I}$ with the property that there exist constants $a,b$
which are strictly positive such that for any vector $u$ in the given Hilbert
space, we have
\begin{equation}
a\left\Vert u\right\Vert ^{2}\leq\sum_{k\in I}\left\vert \left\langle
u,u_{k}\right\rangle \right\vert ^{2}\leq b\left\Vert u\right\Vert
^{2}.\label{definition}%
\end{equation}
The constant values $a,b$ are called the frame bounds of the frame. It can be
derived from (\ref{definition}) that a frame in a finite-dimensional vector
space is simply a spanning set for the vector space. Thus, every basis is a
frame. However, it is not the case that every frame is a basis. Indeed, frames
are generally linearly dependent sets. Let $\pi$ be a unitary representation
of a group $G$ acting in a Hilbert space $\mathbf{H}_{\pi}.$ Let
$v\in\mathbf{H}_{\pi}.$ Any set of the type $\pi(G)v$ which is a frame is
called a $G$-frame.

A frame $\left(  x_{k}\right)  _{k\in I}$ in an $n$-dimensional vector space
is maximally robust to erasures if the removal of any $l\leq m-n$ vectors from
the frame leaves a frame (see \cite{Pfander}, \cite{Casazza} Section $5.$)
Coming back to the example of the Heisenberg group previously discussed, it is
proved in \cite{Pfander} that if $n$ is prime then the set $\left\{
M^{j}T^{k}v:\left(  k,j\right)  \in\mathbb{Z}_{n}^{2}\right\}  $ is maximally
robust to erasures for almost every (with respect to Lebesgue measure)
$v\in\mathbb{C}^{n}.$ In the present work, we consider a variation of this
example. The main objective of this paper is to prove that there exists a
class of Dihedral group frames which are maximally robust to erasures.

Let $n$ be a natural number greater than two. Let $D_{2n}$ be the Dihedral
group of order $2n.$ A presentation of the Dihedral group is:
\[
D_{2n}=\left\langle r,s:r^{n}=s^{2}=1,srs=r^{n-1}\right\rangle .
\]
Next, we define a monomorphism $\kappa:D_{2n}\rightarrow GL\left(
n,\mathbb{\mathbb{C}}\right)  $ such that
\begin{equation}
\kappa\left(  r\right)  =\left(
\begin{array}
[c]{ccccc}%
0 & 0 & \cdots & 0 & 1\\
1 & 0 & \cdots & 0 & 0\\
0 & 1 & \ddots & \vdots & \vdots\\
\vdots & \ddots & \ddots & 0 & 0\\
0 & \cdots & 0 & 1 & 0
\end{array}
\right)  ,\kappa\left(  s\right)  =\left(
\begin{array}
[c]{ccccc}%
1 & 0 & \cdots & 0 & 0\\
0 & 0 & \cdots & 0 & 1\\
\vdots & \vdots & \udots & \udots & 0\\
0 & 0 & 1 & \udots & \vdots\\
0 & 1 & 0 & \cdots & 0
\end{array}
\right)  . \label{kappa}%
\end{equation}
Clearly $\kappa$ is a finite dimensional unitary representation of the
Dihedral group which is reducible. Put%
\[
A=\kappa\left(  r\right)  \text{ and }B=\kappa\left(  s\right)  .
\]
Next, let $\Gamma$ be a finite subgroup of $GL\left(  n,\mathbb{\mathbb{C}%
}\right)  $ which is generated by the matrices $A$ and $B.$ We are interested
in the following questions.\vskip0.5 cm

\noindent\textbf{Problem 1} Let $\alpha$ be a unitary representation of the
Dihedral acting in a Hilbert space $\mathbf{H}_{\alpha}.$ Suppose that
$\alpha$ is unitarily equivalent to $\kappa.$ Under which conditions is it
possible to find a vector $v\in\mathbf{H}_{\alpha}$ such that the set
$\left\{  \alpha\left(  x\right)  v:x\in D_{2n}\right\}  $ has the Haar property?

\noindent\textbf{Problem 2} Let $\tau$ be a unitary irreducible representation
of $D_{2n}$ acting in a finite-dimensional Hilbert space $\mathbf{H}_{\tau}.$
Under which conditions is it possible to find a vector $v\in\mathbf{H}_{\tau}$
such that $\tau\left(  D_{2n}\right)  v=\left\{  \tau\left(  x\right)  v:x\in
D_{2n}\right\}  $ has the Haar property?

To reformulate the problems above, put $D_{2n}=\left\{  x_{1},\cdots
,x_{2n}\right\}  .$ Let $\alpha$ be a representation of the Dihedral
group\ acting in $\mathbf{H}_{\alpha}$ which is either irreducible and unitary
or is equivalent to $\kappa$. We would like to investigate conditions under
which it is possible to find
\[
v=\left(
\begin{array}
[c]{ccc}%
v_{0}, & \cdots & ,v_{\dim\left(  \mathbf{H}_{\alpha}\right)  -1}%
\end{array}
\right)  \in\mathbb{C}^{\dim\left(  \mathbf{H}_{\alpha}\right)  }%
\]
such that every minor of order $\dim\left(  \mathbf{H}_{\alpha}\right)  $ of
the $2n\times\dim\left(  \mathbf{H}_{\alpha}\right)  $ matrix
\[
\left(
\begin{array}
[c]{ccc}%
\left(  x_{1}v\right)  _{0} & \cdots & \left(  x_{1}v\right)  _{\dim
\mathbf{H}_{\alpha}-1}\\
\vdots &  & \vdots\\
\left(  x_{2n}v\right)  _{0} & \cdots & \left(  x_{2n}v\right)  _{\dim
\mathbf{H}_{\alpha}-1}%
\end{array}
\right)
\]
is nonzero. Here is a summary of the main results of the paper.

\begin{theorem}
\label{main result} Let $n$ be a natural number larger than two, and let
$\alpha$ be a representation which is equivalent to $\kappa.$ The following
holds true:

\begin{enumerate}
\item If $n$ is even then it is not possible to find a vector $v\in
\mathbf{H}_{\alpha}$ such that $\alpha\left(  D_{2n}\right)  v$ has the Haar property.

\item If $n$ is prime then there exists a Zariski open set $E\subset
\mathbf{H}_{\alpha}$ such that for any $v\in E,$ $\alpha\left(  D_{2n}\right)
v$ has the Haar property.

\item If $n$ is prime then there exists a Zariski open set $E\subset
\mathbf{H}_{\alpha}$ such that for any $v\in E$, $\alpha\left(  D_{2n}\right)
v$ is a frame in $\mathbb{C}^{n}$ which is maximally robust to erasures.
\end{enumerate}
\end{theorem}

\begin{theorem}
\label{main 3}Let $\tau$ be an irreducible representation of the Dihedral
group acting in a vector space $\mathbf{H}_{\tau}.$

\begin{enumerate}
\item If $\tau$ is a character then for any nonzero complex number $z,$
$\tau\left(  D_{2n}\right)  z$ has the Haar property.

\item If $\tau$ is not a character and if $n$ is prime then for almost every
vector $v$ in $\mathbf{H}_{\tau}$, $\tau\left(  D_{2n}\right)  v$ has the Haar property.

\item If $\tau$ is not a character and if $n$ is even then there does not
exist a vector $v$ in $\mathbf{H}_{\tau}$ such that $\tau\left(
D_{2n}\right)  v$ has the Haar property.

\item Suppose that $n$ is a composite odd natural number.

\begin{enumerate}
\item There exists an irreducible representation $\tau^{\prime}$ of $D_{2n}$
acting in $\mathbf{H}_{\tau^{\prime}}$ such that for any vector in $v$ in
$\mathbf{H}_{\tau^{\prime}}$, $\tau^{\prime}\left(  D_{2n}\right)  v$ does not
have the Haar property.

\item There exists an irreducible representation $\tau^{\prime\prime}$ of
$D_{2n}$ acting in $\mathbf{H}_{\tau^{\prime\prime}}$ such that for almost
every vector in $v$ in $\mathbf{H}_{\tau^{\prime\prime}}$, $\tau^{\prime
\prime}\left(  D_{2n}\right)  v$ has the Haar property.
\end{enumerate}
\end{enumerate}
\end{theorem}

Our work is organized as follows. In the second section, we recall some
well-known facts about Fourier analysis on finite abelian groups, and the
Laplace's Expansion Formulas which are all crucial for the proofs of the main
results. The main results of the paper (Theorem \ref{main result}, and Theorem
\ref{main 3}) are proved in the third section of the paper. Finally, examples
are computed in the fourth section.

\section{Preliminaries}

Let us start by fixing some notations. Given a matrix $M,$ the transpose of $M$
is denoted $M^{T}.$ The determinant of a matrix $M$ is denoted by
$\mathrm{det}(M)$ or $|M|.$ The $k$th row of $M$ is denoted
$\operatorname{row}_{k}\left(  M\right)  $ and similarly, the $k$th column of
the matrix $M$ is denoted $\operatorname{col}_{k}\left(  M\right)  .$

Let $G$ be a group with a binary operation which we denote multiplicatively.
Let $E$ be a subset of $G.$ The set $E^{-1}$ is a subset of $G$ which contains
all inverses of elements of $E.$ More precisely, $E^{-1}=\left\{  a^{-1}:a\in
E\right\}  .$ For example, let $G$ be the cyclic group $\mathbb{Z}_{n}.$ Then
given any subset $E$ of $\mathbb{Z}_{n},$
\[
E^{-1}=\left\{  \left(  n-k\right)  \operatorname{mod}n:k\in E\right\}  .
\]
Let $G$ be a group acting on a set $S.$ We denote this action
multiplicatively. For any fixed element $s\in S,$ the $G$-orbit of $s$ is
described as $Gs=\left\{  gs:g\in G\right\}  .$

Let $\alpha$ be a unitary representation of a group $G$ acting in a Hilbert
space $\mathbf{H}_{\alpha}.$ That is, $\alpha$ is a homomorphism from $G$ into
the group of unitary matrices of order $\dim_{\mathbb{C}}\left(
\mathbf{H}_{\alpha}\right)  .$ We say that $\alpha$ is an irreducible
representation of $G$ if and only if the only subspaces of $\mathbf{H}%
_{\alpha}$ which are invariant under the action of $\alpha$ are the trivial
ones. For example a unitary character (a homomorphism from $G$ into the circle
group) is an irreducible unitary representation. Two unitary representations
$\alpha,\alpha^{\prime}$ of a group $G$ acting in $\mathbf{H}_{\alpha
},\mathbf{H}_{\alpha^{\prime}}$ respectively are equivalent if there exists a
unitary map $U:\mathbf{H}_{\alpha}\rightarrow\mathbf{H}_{\alpha^{\prime}}$
such that
\[
U\alpha\left(  x\right)  U^{-1}=\alpha^{\prime}\left(  x\right)  \text{ for
all }x\in G.
\]
We say that $U$ intertwines the representations $\alpha$ and $\alpha^{\prime
}.$ Let $M$ be a matrix. The transpose of $M$ is denoted $M^T.$ Next, let $z\in\mathbb{C}.$ The complex conjugate of $z$ is written as
$\overline{z}.$ The cardinality of a set $S$ is denoted $\mathrm{card}\left(
S\right)  .$ Throughout this paper, we shall always assume that $n$ is a
natural number larger than two.

\begin{lemma}
\label{intertwines}Let $\alpha,\alpha^{\prime}$ be two equivalent unitary
representations of a group $G.$ Let $U$ be a unitary map which intertwines the
representations $\alpha,\alpha^{\prime}$. Let $v\in$ $\mathbf{H}_{\alpha}.$
Then $\alpha\left(  G\right)  v$ has the Haar property if and only if
$\alpha^{\prime}\left(  G\right)  Uv$ has the Haar property.
\end{lemma}

\begin{proof}
Let us suppose that $\alpha\left(  G\right)  v$ has the Haar property. Then
for any subset $K$ of $G$ of cardinality $\dim\mathbf{H}_{\alpha},$ the set
$\alpha\left(  H\right)  v$ is a basis for the vector space $\mathbf{H}%
_{\alpha}.$ Since $U$ is a unitary map, then $U\alpha\left(  H\right)  v$ is a
basis for $\mathbf{H}_{\alpha^{\prime}}.$ However, $U\alpha\left(  H\right)
v=\alpha^{\prime}\left(  H\right)  Uv.$ Thus, $\alpha^{\prime}\left(
H\right)  Uv$ is a basis for $\mathbf{H}_{\alpha^{\prime}}$ for any subset $H$
of $G$ of cardinality $\dim\mathbf{H}_{\alpha^{\prime}}.$ The proof of the
converse is obtained by using similar arguments, and we shall omit it.
\end{proof}

\subsection{Fourier Analysis on $\mathbb{Z}_{n}$}

Let $\mathbb{Z}_{n}=\left\{  0,1,\cdots,n-1\right\}  .$ We define the Hilbert
space
\[
l^{2}\left(  \mathbb{Z}_{n}\right)  =\left\{  f:\mathbb{Z}_{n}\rightarrow
\mathbb{C}\right\}
\]
which is the set of all complex-valued functions on $\mathbb{Z}_{n}$ endowed
with the following inner product:
\[
\left\langle \phi,\psi\right\rangle =\sum_{x\in\mathbb{Z}_{n}}\phi\left(
x\right)  \overline{\psi\left(  x\right)  }\text{ for }\phi,\psi\in
l^{2}\left(  \mathbb{Z}_{n}\right)  .
\]
The norm of a given vector $\phi$ in $l^{2}\left(  \mathbb{Z}_{n}\right)  $ is
computed as follows:%
\[
\left\Vert \phi\right\Vert _{l^{2}\left(  \mathbb{Z}_{n}\right)
}=\left\langle \phi,\phi\right\rangle ^{1/2}.
\]
We recall that the discrete Fourier transform is a map $\mathcal{F}%
:l^{2}\left(  \mathbb{Z}_{n}\right)  \rightarrow l^{2}\left(  \mathbb{Z}%
_{n}\right)  $ defined by
\[
\left(  \mathcal{F}\phi\right)  \left(  \xi\right)  =\frac{1}{n^{1/2}}%
\sum_{k\in\mathbb{Z}_{n}}\phi\left(  k\right)  \exp\left(  \frac{2\pi ik\xi
}{n}\right)  ,\text{ for }\phi\in l^{2}\left(  \mathbb{Z}_{n}\right)  .
\]
The following facts are also well-known (see \cite{Terras}). Firstly, the
discrete Fourier transform is a bijective linear operator. Secondly, the
Fourier inverse of a vector $\varphi$ is computed as follows:
\[
\mathcal{F}^{-1}\varphi\left(  k\right)  =\frac{1}{n^{1/2}}\sum_{\xi
\in\mathbb{Z}_{n}}\varphi\left(  \xi\right)  \exp\left(  -\frac{2\pi ik\xi}%
{n}\right)  .
\]
Finally, the Fourier transform is a unitary operator. More precisely, given
$\phi,\psi\in l^{2}\left(  \mathbb{Z}_{n}\right)  ,$ we have
\[
\left\langle \phi,\psi\right\rangle =\left\langle \mathcal{F}\phi
,\mathcal{F}\psi\right\rangle .
\]
We shall need the following lemma which is proved in \cite{Evans}.

\begin{lemma}
\label{minor cyclic}Let $\mathbf{F}$ be the matrix representation of the
Fourier transform. If $n$ is prime then every minor of $\mathbf{F}$ is nonzero.
\end{lemma}

We recall that
\[
A=\left(
\begin{array}
[c]{ccccc}%
0 & 0 & \cdots & 0 & 1\\
1 & 0 & \cdots & 0 & 0\\
0 & 1 & \ddots & \vdots & \vdots\\
\vdots & \ddots & \ddots & 0 & 0\\
0 & \cdots & 0 & 1 & 0
\end{array}
\right)  ,\text{ and }B=\left(
\begin{array}
[c]{ccccc}%
1 & 0 & \cdots & 0 & 0\\
0 & 0 & \cdots & 0 & 1\\
\vdots & \vdots & \udots & \udots & 0\\
0 & 0 & 1 & \udots & \vdots\\
0 & 1 & 0 & \cdots & 0
\end{array}
\right)  .
\]
Identifying $l^{2}\left(  \mathbb{Z}_{n}\right)  $ with $%
\mathbb{C}
^{n}$ via the map
\[
v\mapsto\left(
\begin{array}
[c]{ccc}%
v\left(  0\right)  & \cdots & v\left(  n-1\right)
\end{array}
\right)  ^{T},
\]
we may write
\[
Av\left(  j\right)  =v\left(  \left(  j-1\right)  \operatorname{mod}n\right)
\text{ and }Bv\left(  j\right)  =v\left(  \left(  n-j\right)
\operatorname{mod}n\right)  .
\]

\begin{lemma}
\label{FAB}For any $\xi\in\mathbb{Z}_{n},$ we have $\left(  \mathcal{F}%
Bv\right)  \left(  \xi\right)  =\left(  \mathcal{F}v\right)  \left(  \left(
n-\xi\right)  \operatorname{mod}n\right)  $ and $\left(  \mathcal{F}Av\right)
\left(  \xi\right)  =e^{\frac{2\pi i}{n}\xi}\left(  \mathcal{F}v\right)
\left(  \xi\right)  .$
\end{lemma}

\begin{proof}
The proof of this lemma follows from some formal calculations. Firstly,
\begin{align*}
\left(  \mathcal{F}Bv\right)  \left(  \xi\right)   &  =\frac{1}{n^{1/2}}%
\sum_{k\in\mathbb{Z}_{n}}Bv\left(  k\right)  \exp\left(  \frac{2\pi ik\xi}%
{n}\right) \\
&  =\frac{1}{n^{1/2}}\sum_{k\in\mathbb{Z}_{n}}v\left(  \left(  n-k\right)
\operatorname{mod}n\right)  \exp\left(  \frac{2\pi ik\xi}{n}\right) \\
&  =\frac{1}{n^{1/2}}\sum_{k\in\mathbb{Z}_{n}}v\left(  m\right)  \exp\left(
\frac{2\pi im\left(  \left(  n-\xi\right)  \operatorname{mod}n\right)  }%
{n}\right) \\
&  =\left(  \mathcal{F}v\right)  \left(  \left(  n-\xi\right)
\operatorname{mod}n\right)  .
\end{align*}
Secondly,%
\begin{align*}
\mathcal{F}\left(  Av\right)  \left(  \xi\right)   &  =\frac{1}{n^{1/2}}%
\sum_{k\in\mathbb{Z}_{n}}Av\left(  k\right)  \exp\left(  \frac{2\pi ik\xi}%
{n}\right) \\
&  =\frac{1}{n^{1/2}}\sum_{m\in\mathbb{Z}_{n}}v\left(  m\right)  \exp\left(
\frac{2\pi i\left(  m+1\right)  \xi}{n}\right) \\
&  =\frac{1}{n^{1/2}}e^{\frac{2\pi i}{n}\xi}\mathcal{F}v\left(  \xi\right)  .
\end{align*}
This completes the proof.
\end{proof}

From Lemma \ref{FAB}, we obtain the following. Let $\mathbf{F}$ be the matrix
representation of the Fourier transform and define
\[
\mathbf{A}=\mathbf{F}A\mathbf{F}^{-1}\text{ and }\mathbf{B}=\mathbf{F}%
B\mathbf{F}^{-1}=B.
\]
Then
\[
\mathbf{A}=\left(
\begin{array}
[c]{ccccc}%
1 &  &  &  & \\
& e^{\frac{2\pi i}{n}} &  &  & \\
&  & e^{\frac{4\pi i}{n}} &  & \\
&  &  & \ddots & \\
&  &  &  & e^{\frac{2\pi i\left(  n-1\right)  }{n}}%
\end{array}
\right)  \text{ and }\mathbf{B=}\left(
\begin{array}
[c]{ccccc}%
1 & 0 & \cdots & 0 & 0\\
0 & 0 & \cdots & 0 & 1\\
\vdots & \vdots & \udots & \udots & 0\\
0 & 0 & 1 & \udots & \vdots\\
0 & 1 & 0 & \cdots & 0
\end{array}
\right)  .
\]

\subsubsection{Laplace's Expansion Theorem}

The following discussion is mainly taken from Chapter $3,$ \cite{Howard}. Let
$X$ be a square matrix of order $n.$

\begin{definition}
A minor of $X$ is the determinant of any square sub-matrix $Y$ of $X.$ Let
$\left\vert Y\right\vert $ be an $m$-rowed minor of $X.$ The determinant of
the sub-matrix obtained by deleting from $X$ the rows and columns represented
in $Y$ is called the complement of $\left\vert Y\right\vert .$ Let $\left\vert
Y\right\vert $ be the $m$-rowed minor of $X$ in which rows $i_{1},\cdots
,i_{m}$ and columns $j_{1},\cdots,j_{m}$ are represented. Then the algebraic
complement, or cofactor of $\left\vert Y\right\vert $ is given by
\[
\left(  -1\right)  ^{\sum_{k=1}^{m}i_{k}+\sum_{k=1}^{m}j_{k}}\left\vert
Z\right\vert
\]
where $\left\vert Z\right\vert $ is the complent of $\left\vert Y\right\vert
.$
\end{definition}

According to \textbf{Laplace's Expansion Theorem} (see $3.7.3,$ \cite{Howard})
a formula for the determinant of $X$ can be obtained by following three main steps.

\begin{enumerate}
\item Select any $m$ rows (or columns) from the matrix $X.$

\item Collect all $m$-rowed minors of $X$ found in these $m$ rows (or columns).

\item The determinant of $X$ is equal to the sum of the products of each of
these minors and its algebraic complement.
\end{enumerate}

To be more precise, let $X=\left(  X_{i,j}\right)  _{1\leq i,j\leq n}$ be a
square matrix of order $n.$ Let $T\left(  n,p\right)  $ be the set of all
$p$-tuples of integers: $s=\left(  s_{1},\cdots,s_{p}\right)  $ where $1\leq
s_{1}<s_{2}<\cdots<s_{p}\leq n.$ Given any $s,t\in T\left(  n,p\right)  ,$ we
let $X\left(  s,t\right)  $ be the sub-matrix of order $p$ of $A$ such that
\[
X\left(  s,t\right)  _{i,j}=X_{s_{i},t_{j}}.
\]
Next, let $X\left(  s,t\right)  ^{c}$ be the complementary matrix of $X\left(
s,t\right)  $ which is a matrix of order $n-p$ obtained by removing rows
$s_{1},\cdots,s_{p}$ and columns $t_{1},\cdots,t_{p}$ from the matrix $A.$
Define
\[
\left\vert s\right\vert =\sum_{k=1}^{p}s_{k}.
\]
According to Laplace's Expansion Theorem, for any fixed
$t\in T\left(  n,p\right)  ,$
\begin{equation}
\det\left(  X\right)  =\sum_{s\in T\left(  n,p\right)  }\left(  -1\right)
^{\left\vert s\right\vert +\left\vert t\right\vert }\det\left(  X\left(
s,t\right)  \right)  \det\left(  X\left(  s,t\right)  ^{c}\right)  .
\label{Laplace}%
\end{equation}

\section{Proof of Main Results}

\begin{proposition}
\label{even}Assume that $n>2$ and is even. Given any $\gamma\in\Gamma,$ the
following holds true:%
\[
\sum_{k=0}^{\frac{n-2}{2}}\gamma A^{2k}=\sum_{k=0}^{\frac{n-2}{2}}\gamma
A^{2k}B.
\]
In other words, there exists a subset $\left\{  \gamma_{k_{1}},\cdots
,\gamma_{k_{n}}\right\}  $ of $\Gamma$ of cardinality $n$ which is linearly
dependent over $\mathbb{C}.$
\end{proposition}

\begin{proof}
Let $I=\left\{  0,2,\cdots,n-2\right\}  .$ Put $\omega=e^{\frac{2\pi i}{n}}.$
Then
\[
\sum_{k\in I}\mathbf{A}^{k}=\sum_{k=0}^{\frac{n-2}{2}}\mathbf{A}^{2k}=\left(
\begin{array}
[c]{cccc}%
\sum_{k=0}^{\frac{n-2}{2}}1 &  &  & \\
& \sum_{k=0}^{\frac{n-2}{2}}\omega^{2k} &  & \\
&  & \ddots & \\
&  &  & \sum_{k=0}^{\frac{n-2}{2}}\omega^{2\left(  n-1\right)  k}%
\end{array}
\right)  .
\]
Now, we claim that $\sum_{k\in I}\mathbf{A}^{k}$ is a diagonal matrix with
only two nonzero entry. To see that this holds, it suffices to observe that
$\left(  \sum_{k=0}^{\frac{n-2}{2}}\mathbf{A}^{2k}\right)  _{1,1}=\frac{n}{2}$
and for $j\neq1,$%
\[
\left(  \sum_{k=0}^{\frac{n-2}{2}}\mathbf{A}^{2k}\right)  _{j,j}=\sum
_{k=0}^{\frac{n-2}{2}}\omega^{2jk}=\left\{
\begin{array}
[c]{ccc}%
0 & \text{if} & j\neq\frac{n}{2}\\
\frac{n}{2} & \text{if} & j=\frac{n}{2}%
\end{array}
\right.  .
\]
Moreover, for $j\neq l,$\ $\left(  \sum_{k=0}^{\frac{n-2}{2}}\mathbf{A}%
^{2k}\right)  _{j,l}=0.$ Next
\begin{align}
\sum_{k\in I}\mathbf{A}^{k}\mathbf{B}  &  \mathbf{=}\left(
\begin{array}
[c]{ccccc}%
\frac{n}{2} &  &  &  & \\
& \sum_{k=0}^{\frac{n-2}{2}}e^{\frac{2\pi i\left(  2k\right)  }{n}} & 0 &
\cdots & 0\\
& 0 & \sum_{k=0}^{\frac{n-2}{2}}e^{\frac{2\pi i\left(  2k\right)  2}{n}} &  &
\vdots\\
& \vdots &  & \ddots & 0\\
& 0 & \cdots & 0 & \sum_{k=0}^{\frac{n-2}{2}}e^{\frac{2\pi i\left(  2k\right)
\left(  n-1\right)  }{n}}%
\end{array}
\right)  \left(
\begin{array}
[c]{ccccc}%
1 &  &  &  & \\
& 0 & 0 & \cdots & 1\\
& \vdots &  & \udots & 0\\
& 0 & 1 &  & \vdots\\
& 1 & 0 & \cdots & 0
\end{array}
\right) \label{line 1}\\
&  =\left(
\begin{array}
[c]{ccccc}%
\frac{n}{2} &  &  &  & \\
& 0 & 0 & \cdots & \sum_{k=0}^{\frac{n-2}{2}}e^{\frac{2\pi i\left(  2k\right)
}{n}}\\
& \vdots &  & \udots & 0\\
& 0 & \sum_{k=0}^{\frac{n-2}{2}}e^{\frac{2\pi i\left(  2k\right)  \left(
n-2\right)  }{n}} &  & \vdots\\
& \sum_{k=0}^{\frac{n-2}{2}}e^{\frac{2\pi i\left(  2k\right)  \left(
n-1\right)  }{n}} & 0 & \cdots & 0
\end{array}
\right)  . \label{line 2}%
\end{align}
From (\ref{line 1}), and (\ref{line 2}), it is easy to see that the entry
\[
\sum_{k=0}^{\frac{n-2}{2}}e^{\frac{2\pi i\left(  2k\right)  \left(
n-j\right)  }{n}}\text{ for }1\leq j\leq n-1
\]
is the only possible nonzero element of $\operatorname{col}_{j+1}\left(  \sum_{k\in
I}\mathbf{A}^{k}\mathbf{B}\right)  .$ Thus, for any index $j,$ ($0\leq j\leq
n-1$) the complex number $\sum_{k=0}^{\frac{n-2}{2}}e^{\frac{2\pi i\left(
2k\right)  \left(  n-j\right)  }{n}}$ is a diagonal entry of the matrix
$\sum_{k\in I}\mathbf{A}^{k}\mathbf{B}$ if and only if $j=\frac{n}{2},$ or
$j=0.$ Therefore, $\sum_{k=0}^{\frac{n-2}{2}}\mathbf{A}^{2k}=\sum_{k=0}%
^{\frac{n-2}{2}}\mathbf{A}^{2k}\mathbf{B}$ and this implies that $\sum
_{k=0}^{\frac{n-2}{2}}\mathbf{F}^{-1}\mathbf{A}^{2k}=\sum_{k=0}^{\frac{n-2}%
{2}}\mathbf{F}^{-1}\mathbf{A}^{2k}\mathbf{B.}$ Since $\mathbf{F}%
^{-1}\mathbf{A}=A\mathbf{F}^{-1}$ and $\mathbf{F}^{-1}\mathbf{B}%
=B\mathbf{F}^{-1}$ then
\begin{equation}
\sum_{k=0}^{\frac{n-2}{2}}A^{2k}=\sum_{k=0}^{\frac{n-2}{2}}A^{2k}B.
\label{abov}%
\end{equation}
Finally, given any $\gamma\in\Gamma,$ by multiplying (\ref{abov}) on the left
by $\gamma$ we obtain the desired result.
\end{proof}

\begin{corollary}
\label{even case}If $n$ is even then it is not possible to find a vector
$v\in\mathbb{C}^{n}$ such that $\Gamma v$ has the Haar property.
\end{corollary}

\begin{proof}
According to Proposition \ref{even}, any vector $v$ is in the kernel of the
linear operator $\sum_{k=0}^{\frac{n-2}{2}}A^{2k}-\sum_{k=0}^{\frac{n-2}{2}%
}A^{2k}B.$ Thus, for any fixed vector $v\in\mathbb{C}^{n},$ the set of
vectors
\[
\left\{  A^{2k}v:0\leq k\leq\frac{n-2}{2}\right\}  \cup\left\{  A^{2k}Bv:0\leq
k\leq\frac{n-2}{2}\right\}
\]
is linearly dependent.
\end{proof}

\begin{lemma}
\label{impossible}Assume that $n$ is an odd natural number greater than
one.\ Let $m\in\mathbb{N}$ such that $1\leq m<n.$ Let $\mathbb{Z}_{n}=\left\{
0,1,\cdots,m-1\right\}  \cup\left\{  m,\cdots,n-1\right\}  $. Let
$B_{1}\subseteq\left\{  0,1,\cdots,m-1\right\}  ,B_{2}\subseteq\left\{
m,\cdots,n-1\right\}  $ such that $\mathrm{card}\left(  B_{1}\right)
=\mathrm{card}\left(  B_{2}\right)  \geq1.$ Then it is not possible for
$B_{1}^{-1}=B_{1}$ and $B_{2}^{-1}=B_{2}.$
\end{lemma}

\begin{proof}
We shall prove this lemma by cases. For the first case, let us suppose that
$\mathrm{card}\left(  B_{1}\right)  =\mathrm{card}\left(  B_{2}\right)  =1.$
Since $B_{1}$ and $B_{2}$ are disjoint, then either $B_{1}$ contains a
non-trivial element or $B_{2}$ contains a non-trivial element. In either case,
it is not possible for $B_{1}^{-1}=B_{1}$ and $B_{2}^{-1}=B_{2}.$ This is due
to the fact that when $n$ is odd, the only element which is equal to its
additive inverse ($\operatorname{mod}n$) is the trivial element $0.$ For the
second case, let us suppose that $m<\frac{n}{2}$ and $\mathrm{card}\left(
B_{1}\right)  =\mathrm{card}\left(  B_{2}\right)  >1.$ Then, there is at least
one non-trivial element of $\mathbb{Z}_{n}$ in the set $B_{1}.$ If $B_{1}%
^{-1}=B_{1}$ then there exist $k,k^{\prime}\in B_{1}$ such that $k=n-k^{\prime
}.$ Now, since $k,k^{\prime}\leq m-1$ then $n=k+k^{\prime}\leq2\left(
m-1\right)  <n-2$ and this is absurd. For the third case, let us suppose that
$m>\frac{n}{2}$ and $\mathrm{card}\left(  B_{1}\right)  =\mathrm{card}\left(
B_{2}\right)  >1.$ If $B_{2}^{-1}=B_{2}$ then there must exist $k,k^{\prime
}\in B_{2}$ such that $k+k^{\prime}=n,$ and $k,k^{\prime}\geq m.$ Thus,
$n=k+k^{\prime}\geq2m>n$ and this is absurd as well.
\end{proof}

\begin{example}
Let $\mathbb{Z}_{7}=\left\{  0,\cdots,6\right\}  .$ Put $m=3.$ Now let
$B_{1}=\left\{  0,1\right\}  \text{ and }B_{2}=\left\{  3,4\right\}  .$ Then
$B_{2}^{-1}=B_{2}.$ However, $B_{1}^{-1}=\left\{  0,6\right\}  \neq B_{1}.$
\end{example}

\begin{remark}
We remark here that Lemma \ref{impossible} fails when $n$ is even. For
example, let us consider the finite cyclic group of order four. Let $m=2,$
$B_{1}=\left\{  0\right\}  $ and $B_{2}=\left\{  2\right\}  .$ Then clearly,
$B_{1}^{-1}=B_{1}$ and $B_{2}^{-1}=B_{2}.$
\end{remark}

Define the group%
\[
\Sigma=\mathbf{F}\Gamma\mathbf{F}^{-1}=\left\{  \mathbf{F}\gamma
\mathbf{F}^{-1}:\gamma\in\Gamma\right\}
\]
which is also isomorphic to the Dihedral group. We recall that for any vector
$v\in%
\mathbb{C}
^{n},$ we write
\[
v=\left(
\begin{array}
[c]{cccc}%
v_{0} & v_{1} & \cdots & v_{n-1}%
\end{array}
\right)  ^{T}.
\]
For any subset $\Lambda=\left\{  \gamma_{k_{1}},\cdots,\gamma_{k_{n}}\right\}
$ of the group $\Sigma$, we consider the corresponding matrix-valued function
defined on $\mathbf{\mathbb{C}}^{n}$ as follows.
\[
\delta_{\Lambda}:f\mapsto\left(
\begin{array}
[c]{c}%
\gamma_{k_{1}}f\\
\vdots\\
\gamma_{k_{n}}f
\end{array}
\right)  =\left(
\begin{array}
[c]{ccc}%
\left(  \gamma_{k_{1}}f\right)  _{0} & \cdots & \left(  \gamma_{k_{1}%
}f\right)  _{n-1}\\
\vdots & \ddots & \vdots\\
\left(  \gamma_{k_{n}}f\right)  _{0} & \cdots & \left(  \gamma_{k_{n}%
}f\right)  _{n-1}%
\end{array}
\right)  .
\]
We acknowledge that the proof of the following proposition was partly inspired
by the proof given for Theorem $4$ \cite{Pfander}.

\begin{proposition}
\label{prime case}Let $\Lambda$ be any subset of $\Sigma$ of cardinality $n.$
If $n$ is prime then there exists a Zariski open set $E\subset\mathbb{C}^{n}$
such that given any vector $f\in E,$ $\det\delta_{\Lambda}\left(  f\right)  $
is a non-vanishing homogeneous polynomial.
\end{proposition}

\begin{proof}
Put $\omega=e^{\frac{2\pi i}{n}}.$ There are several cases to consider. For
the first case, let us suppose that there exist natural numbers $m,p$ such
that $m+p=n$, such that%
\[
\Lambda=\left\{  \mathbf{A}^{k_{1}},\cdots,\mathbf{A}^{k_{m}},\mathbf{A}%
^{\ell_{1}}\mathbf{B},\cdots,\mathbf{A}^{\ell_{p}}\mathbf{B}\right\}
\]
and
\begin{equation}
\delta_{\Lambda}\left(  f\right)  =\left(
\begin{array}
[c]{cccccccc}%
f_{0} & \omega^{k_{1}}f_{1} & \cdots & \omega^{\left(  m-1\right)  k_{1}%
}f_{m-1} & \omega^{mk_{1}}f_{m} & \cdots & \omega^{\left(  n-2\right)  k_{1}%
}f_{n-2} & \omega^{\left(  n-1\right)  k_{1}}f_{n-1}\\
f_{0} & \omega^{k_{2}}f_{1} & \cdots & \omega^{\left(  m-1\right)  k_{2}%
}f_{m-1} & \omega^{mk_{2}}f_{m} & \cdots & \omega^{\left(  n-2\right)  k_{2}%
}f_{n-2} & \omega^{\left(  n-1\right)  k_{2}}f_{n-1}\\
\vdots & \vdots &  & \vdots & \vdots &  & \vdots & \vdots\\
f_{0} & \omega^{k_{m}}f_{1} & \cdots & \omega^{\left(  m-1\right)  k_{m}%
}f_{m-1} & \omega^{mk_{m}}f_{m} & \cdots & \omega^{\left(  n-2\right)  k_{m}%
}f_{n-2} & \omega^{\left(  n-1\right)  k_{m}}f_{n-1}\\
f_{0} & \omega^{\ell_{1}}f_{n-1} & \cdots & \omega^{\left(  m-1\right)
\ell_{1}}f_{n-\left(  m-1\right)  } & \omega^{m\ell_{1}}f_{n-m} & \cdots &
\omega^{\left(  n-2\right)  \ell_{1}}f_{2} & \omega^{\left(  n-1\right)
\ell_{1}}f_{1}\\
f_{0} & \omega^{\ell_{2}}f_{n-1} & \cdots & \omega^{\left(  m-1\right)
\ell_{2}}f_{n-\left(  m-1\right)  } & \omega^{m\ell_{2}}f_{n-m} & \cdots &
\omega^{\left(  n-2\right)  \ell_{2}}f_{2} & \omega^{\left(  n-1\right)
\ell_{2}}f_{1}\\
\vdots & \vdots &  & \vdots & \vdots &  & \vdots & \vdots\\
f_{0} & \omega^{\ell_{p}}f_{n-1} & \cdots & \omega^{\left(  m-1\right)
\ell_{p}}f_{n-\left(  m-1\right)  } & \omega^{m\ell_{p}}f_{n-m} & \cdots &
\omega^{\left(  n-2\right)  \ell_{p}}f_{2} & \omega^{\left(  n-1\right)
\ell_{p}}f_{1}%
\end{array}
\right)  . \label{first case}%
\end{equation}
Now, fix $t=\left(  1,\cdots,m\right)  .$ We consider the transpose of
$\delta_{\Lambda}\left(  f\right)  $ which is given by
\[
\left(
\begin{array}
[c]{cccccc}%
f_{0} & \cdots & f_{0} & f_{0} & \cdots & f_{0}\\
\omega^{k_{1}}f_{1} & \cdots & \omega^{k_{m}}f_{1} & \omega^{\ell_{1}}f_{n-1}
& \cdots & \omega^{\ell_{p}}f_{n-1}\\
\vdots &  & \vdots & \vdots &  & \vdots\\
\omega^{\left(  m-1\right)  k_{1}}f_{m-1} & \cdots & \omega^{\left(
m-1\right)  k_{m}}f_{m-1} & \omega^{\left(  m-1\right)  \ell_{1}}f_{n-\left(
m-1\right)  } & \cdots & \omega^{\left(  m-1\right)  \ell_{p}}f_{n-\left(
m-1\right)  }\\
\omega^{mk_{1}}f_{m} & \cdots & \omega^{mk_{m}}f_{m} & \omega^{m\ell_{1}%
}f_{n-m} & \cdots & \omega^{m\ell_{p}}f_{n-m}\\
\vdots &  & \vdots & \vdots &  & \vdots\\
\omega^{\left(  n-2\right)  k_{1}}f_{n-2} & \cdots & \omega^{\left(
n-2\right)  k_{m}}f_{n-2} & \omega^{\left(  n-2\right)  \ell_{1}}f_{2} &
\cdots & \omega^{\left(  n-2\right)  \ell_{p}}f_{2}\\
\omega^{\left(  n-1\right)  k_{1}}f_{n-1} & \cdots & \omega^{\left(
n-1\right)  k_{m}}f_{n-1} & \omega^{\left(  n-1\right)  \ell_{1}}f_{1} &
\cdots & \omega^{\left(  n-1\right)  \ell_{p}}f_{1}%
\end{array}
\right)  .
\]
To avoid cluster of notation, put
\[
M_{f}=\left(  \delta_{\Lambda}\left(  f\right)  \right)  ^{T}.
\]
Applying Laplace's Expansion Theorem (\ref{Laplace}) to $M_{f},$ we obtain
\[
\det\left(  M_{f}\right)  =\sum_{s\in T\left(  n,m\right)  }\left(  -1\right)
^{\left\vert s\right\vert +\left\vert t\right\vert }\det\left(  M_{f}\left(
s,t\right)  \right)  \det\left(  \left(  M_{f}\left(  s,t\right)  \right)
^{c}\right)  .
\]
For $t=\left(  1,\cdots,m\right)  ,$ $M_{f}\left(  t,t\right)  $ is the matrix
obtained by retaining the first $m$ rows and first $m$ columns of the matrix
$M_{f}.$ The matrix $M_{f}\left(  t,t\right)  ^{c}$ is a matrix of order
$n-m=p$ which is obtained by deleting the first $m$ rows and the first $m$
columns of $M_{f}.$ Thus, for $t=\left(  1,\cdots,m\right)  ,$ it is easy to
see that $\left(  -1\right)  ^{2\left\vert t\right\vert }\det\left(
M_{f}\left(  t,t\right)  \right)  \det\left(  M_{f}\left(  t,t\right)
^{c}\right)  $ is equal to
\begin{equation}
p_{\Lambda}\left(  f\right)  =\left\vert
\begin{array}
[c]{ccc}%
f_{0} & \cdots & f_{0}\\
\omega^{k_{1}}f_{1} & \cdots & \omega^{k_{m}}f_{1}\\
\vdots &  & \vdots\\
\omega^{\left(  m-1\right)  k_{1}}f_{m-1} & \cdots & \omega^{\left(
m-1\right)  k_{m}}f_{m-1}%
\end{array}
\right\vert \left\vert
\begin{array}
[c]{ccc}%
\omega^{m\ell_{1}}f_{n-m} & \cdots & \omega^{m\ell_{p}}f_{n-m}\\
\vdots &  & \vdots\\
\omega^{\left(  n-2\right)  \ell_{1}}f_{2} & \cdots & \omega^{\left(
n-2\right)  \ell_{p}}f_{2}\\
\omega^{\left(  n-1\right)  \ell_{1}}f_{1} & \cdots & \omega^{\left(
n-1\right)  \ell_{p}}f_{1}%
\end{array}
\right\vert . \label{monomial}%
\end{equation}
Using the fact that the determinant map is multi-linear, then (\ref{monomial})
becomes%
\begin{equation}
p_{\Lambda}\left(  f\right)  =ar\left(  f\right)  \label{rf}%
\end{equation}
where
\[
a=\left\vert
\begin{array}
[c]{ccc}%
1 & \cdots & 1\\
\omega^{k_{1}} & \cdots & \omega^{k_{m}}\\
\vdots &  & \vdots\\
\omega^{\left(  m-1\right)  k_{1}} & \cdots & \omega^{\left(  m-1\right)
k_{m}}%
\end{array}
\right\vert \left\vert
\begin{array}
[c]{ccc}%
\omega^{m\ell_{1}} & \cdots & \omega^{m\ell_{p}}\\
\vdots &  & \vdots\\
\omega^{\left(  n-2\right)  \ell_{1}} & \cdots & \omega^{\left(  n-2\right)
\ell_{p}}\\
\omega^{\left(  n-1\right)  \ell_{1}} & \cdots & \omega^{\left(  n-1\right)
\ell_{p}}%
\end{array}
\right\vert \in%
\mathbb{C}
,
\]
and $r\left(  f\right)  $ is the monomial given by
\begin{equation}
r\left(  f\right)  ={\displaystyle\prod\limits_{k=0}^{m-1}}f_{k}%
{\displaystyle\prod\limits_{j=1}^{n-m}}f_{j}. \label{mono}%
\end{equation}
Using Chebotarev's theorem (see Lemma \ref{minor cyclic}), since $n$ is prime
and because $a$ is a product of minors of the discrete Fourier matrix, then
$a\neq0$ and the polynomial $p_{\Lambda}\left(  f\right)  $ is nonzero. Next,
we remark that $\det\left(  M_{f}\right)  $ is a homogeneous polynomial of
degree $n$ in the variables $f_{0},\cdots,f_{n-1}$ and can be uniquely written
as
\[
\det\left(  M_{f}\right)  =\sum_{\alpha\in\mathbb{Z}_{+}^{n},\left\vert
\alpha\right\vert =n}a_{\alpha}f_{0}^{\alpha_{0}}\cdots f_{n-1}^{\alpha_{n-1}%
},\text{ where }a_{\alpha}\in\mathbb{C}.
\]
Regarding the formula above, we remind the reader that the multi-index
$\alpha$ is equal to $\left(  \alpha_{0},\cdots,\alpha_{n-1}\right)  .$ To
show that the polynomial $\det\left(  M_{f}\right)  $ is nonzero, it suffices
to find a multi-index $\alpha$ such that $\left\vert \alpha\right\vert =n$ and
$a_{\alpha}\neq0.$ In order to prove this fact, we would like to isolate a
certain monomial of the type $f_{0}^{\alpha_{0}}\cdots f_{n-1}^{\alpha_{n-1}}$
in
\begin{equation}
\det\left(  M_{f}\right)  =\sum_{s\in T\left(  n,m\right)  }\left(  -1\right)
^{\left\vert s\right\vert +\left\vert t\right\vert }\det\left(  M_{f}\left(
s,t\right)  \right)  \det\left(  \left(  M_{f}\left(  s,t\right)  \right)
^{c}\right)  \label{determinant}%
\end{equation}
and prove that its corresponding coefficient $a_{\alpha}$ is non zero. The
monomial in question that we aim to isolate is $r\left(  f\right)  $ which is
defined in (\ref{mono}). We shall prove that the corresponding coefficient in
(\ref{determinant}) to $r\left(  f\right)  $ is just the complex number $a$
which is described in Formula (\ref{rf}). First, it is easy to see that
\[
r\left(  f\right)  ={\displaystyle\prod\limits_{k\in I\left(  s\right)  }%
}f_{k}{\displaystyle\prod\limits_{j\in I\left(  s\right)  ^{c}}}f_{j}%
\]
where
\[
I\left(  s\right)  =\left\{  0\leq k\leq m-1\right\}  \text{ and }I\left(
s\right)  ^{c}=\left\{  1\leq k\leq n-m\right\}  .
\]
Next for any $s^{\circ}\in T\left(  n,m\right)  $, let us suppose that
$s^{\circ}\neq t.$ We may write $s^{\circ}=\left(  s_{j_{1}}^{\circ}%
,\cdots,s_{j_{m}}^{\circ}\right)  \ $and
\[
\det\left(  M_{f}\left(  s^{\circ},t\right)  \right)  \det\left(  M_{f}\left(
s^{\circ},t\right)  ^{c}\right)  =a^{\circ}{\displaystyle\prod\limits_{k\in
I\left(  s^{\circ}\right)  }}f_{k}{\displaystyle\prod\limits_{I\left(
s^{\circ}\right)  ^{c}}}f_{k}%
\]
where $a^{\circ}\in\mathbb{C}$ and the sets $I\left(  s^{\circ}\right)  $ and
$I\left(  s^{\circ}\right)  ^{c}$ are described as follows. There exists a
natural number $m_{1}\leq m$ such that
\begin{equation}
I\left(  s^{\circ}\right)  =\left(  I\left(  s\right)  -\left\{  j_{1}%
,\cdots,j_{m_{1}}\right\}  \right)  \cup\left\{  j_{1}^{\circ},\cdots
,j_{m_{1}}^{\circ}\right\}  , \label{new1}%
\end{equation}
all the $j_{k}^{\circ}\ $are greater or equal to $m,$ all the $j_{k}\ $are
less or equal to $m-1,$ $\left\{  j_{1}^{\circ},\cdots,j_{m_{1}}^{\circ
}\right\}  \cap\left\{  j_{1},\cdots,j_{m_{1}}\right\}  $ is a null set and
\begin{equation}
I\left(  s^{\circ}\right)  ^{c}=\left(  I\left(  s\right)  ^{c}-\left\{
\overline{n-j_{1}^{\circ}},\cdots,\overline{n-j_{m_{1}}^{\circ}}\right\}
\right)  \cup\left\{  \overline{n-j_{1}},\cdots,\overline{n-j_{m_{1}}%
}\right\}  . \label{new 2}%
\end{equation}
Here $\overline{x}$ stands for $x\operatorname{mod}n.$ The set $\left\{
j_{1},\cdots,j_{m_{1}}\right\}  $ corresponds to the set of rows removed from
$M_{f}\left(  s,t\right)  $ and the set $\left\{  j_{1}^{\circ},\cdots
,j_{m_{1}}^{\circ}\right\}  $ corresponds to the new rows which are then added
to form a new sub-matrix $M_{f}\left(  s^{\circ},t\right)  .$ To prove that
the coefficient $a$ is the unique coefficient of the monomial $r(f),$ let us
assume by contradiction that there exists $s^{\circ}\neq\left(  1,\cdots
,m\right)  $ such that its corresponding monomial in (\ref{determinant}) is
\[
{\prod\limits_{k\in I\left(  s^{\circ}\right)  }}f_{k}{\prod\limits_{j\in
I\left(  s^{\circ}\right)  ^{c}}}f_{j}={\prod\limits_{k\in I\left(  s\right)
}}f_{k}{\prod\limits_{j\in I\left(  s\right)  ^{c}}}f_{j}.
\]
Appealing to (\ref{new1}) and (\ref{new 2}) we have
\begin{align*}
{\prod\limits_{k\in I\left(  s^{\circ}\right)  }}f_{k}{\prod\limits_{I\left(
s^{\circ}\right)  ^{c}}}f_{k}  &  =\left(  {\prod\limits_{k\in\left(  I\left(
s\right)  -\left\{  j_{1},\cdots,j_{m_{1}}\right\}  \right)  \cup\left\{
j_{1}^{\circ},\cdots,j_{m_{1}}^{\circ}\right\}  }}f_{k}\right)  \left(
{\prod\limits_{j\in\left(  I\left(  s\right)  ^{c}-\left\{  \overline
{n-j_{1}^{\circ}},\cdots,\overline{n-j_{m_{1}}^{\circ}}\right\}  \right)
\cup\left\{  \overline{n-j_{1}},\cdots,\overline{n-j_{m_{1}}}\right\}  }}%
f_{j}\right) \\
&  =\left(  f_{0}\cdots f_{m-1}\right)  \left(  f_{1}\cdots f_{n-m}\right)  .
\end{align*}
Since $\left\{  j_{1}^{\circ},\cdots,j_{m_{1}}^{\circ}\right\}  \cap\left\{
j_{1},\cdots,j_{m_{1}}\right\}  $ is an empty set then it must be the case
that
\[
f_{j_{1}^{\circ}}\cdots f_{j_{m_{1}}^{\circ}}=f_{\overline{n-j_{1}^{\circ}}%
}\cdots f_{\overline{n-j_{m_{1}}^{\circ}}}\text{ and }f_{\overline{n-j_{1}}%
}\cdots f_{\overline{n-j_{m_{1}}}}=f_{j_{1}}\cdots f_{j_{m_{1}}}.
\]
As a result,
\begin{equation}
\left\{  j_{1},\cdots,j_{m_{1}}\right\}  =\left\{  \overline{n-j_{1}}%
,\cdots,\overline{n-j_{m_{1}}}\right\}  ,\text{ }\left\{  j_{1}^{\circ}%
,\cdots,j_{m_{1}}^{\circ}\right\}  =\left\{  \overline{n-j_{1}^{\circ}}%
,\cdots,\overline{n-j_{m_{1}}^{\circ}}\right\}  . \label{sets}%
\end{equation}
We observe that equality (\ref{sets}) is equivalent to
\[
\left\{  j_{1},\cdots,j_{m_{1}}\right\}  =\left\{  j_{1},\cdots,j_{m_{1}%
}\right\}  ^{-1}\text{ and }\left\{  j_{1}^{\circ},\cdots,j_{m_{1}}^{\circ
}\right\}  =\left\{  j_{1}^{\circ},\cdots,j_{m_{1}}^{\circ}\right\}  ^{-1}.
\]
Now using the fact that
\[
\max\left(  \left\{  j_{1},\cdots,j_{m_{1}}\right\}  \right)  \leq m-1\text{
and }\min\left(  \left\{  j_{1}^{\circ},\cdots,j_{m_{1}}^{\circ}\right\}
\right)  \geq m,
\]
together with Lemma \ref{impossible}, then statement (\ref{sets}) is absurd.
Thus, the corresponding coefficient in (\ref{determinant}) to the monomial
$r\left(  f\right)  $ is the nonzero complex number $a.$ So, if
\[
\Lambda=\left\{  \mathbf{A}^{k_{1}},\cdots,\mathbf{A}^{k_{m}},\mathbf{A}%
^{\ell_{1}}\mathbf{B},\cdots,\mathbf{A}^{\ell_{p}}\mathbf{B}\right\}
\]
then $\det\left(  \delta_{\Lambda}\left(  f\right)  \right)  $ is a
non-vanishing polynomial. This completes the proof for the first case. For the
other remaining cases, we have two other possibilities to consider. Either
$\Lambda=\left\{  \mathbf{A}^{k_{1}},\cdots,\mathbf{A}^{_{k_{n}}}\right\}  $
or $\Lambda=\left\{  \mathbf{A}^{k_{1}}\mathbf{B},\cdots,\mathbf{A}^{k_{n}%
}\mathbf{B}\right\}  .$ Let us suppose that $\Lambda=\left\{  \mathbf{A}%
^{k_{1}},\cdots,\mathbf{A}^{_{k_{n}}}\right\}  .$ Put
\[
a^{\prime}=\left\vert
\begin{array}
[c]{ccccc}%
1 & \omega^{k_{1}} & \cdots & \omega^{\left(  n-2\right)  k_{1}} &
\omega^{\left(  n-1\right)  k_{1}}\\
1 & \omega^{k_{2}} & \cdots & \omega^{\left(  n-2\right)  k_{2}} &
\omega^{\left(  n-1\right)  k_{2}}\\
\vdots & \vdots & \ddots & \vdots & \vdots\\
1 & \omega^{k_{n}} & \cdots & \omega^{\left(  n-2\right)  k_{n}} &
\omega^{\left(  n-1\right)  k_{n}}%
\end{array}
\right\vert .
\]
Then%
\[
\det\left(  \delta_{\Lambda}\left(  f\right)  \right)  =a^{\prime
}{\displaystyle\prod\limits_{j=0}^{n-1}}f_{j}.
\]
Appealing again to the fact that $n$ is prime, and since $a^{\prime}$ is a
minor of a Fourier matrix then $\det\left(  \delta_{\Lambda}\left(  f\right)
\right)  $ is also a non-vanishing polynomial. For the last case, let us
assume that $\Lambda=\left\{  \mathbf{A}^{k_{1}}\mathbf{B},\cdots
,\mathbf{A}^{k_{n}}\mathbf{B}\right\}  .$ Then
\[
\delta_{\Lambda}\left(  f\right)  =\left(
\begin{array}
[c]{cccc}%
f_{0} & \omega^{\ell_{1}}f_{n-1} & \cdots & \omega^{\left(  n-1\right)
\ell_{1}}f_{1}\\
f_{0} & \omega^{\ell_{2}}f_{n-1} & \cdots & \omega^{\left(  n-1\right)
\ell_{2}}f_{1}\\
\vdots & \vdots &  & \vdots\\
f_{0} & \omega^{\ell_{n-1}}f_{n-1} & \cdots & \omega^{\left(  n-1\right)
\ell_{p}}f_{1}%
\end{array}
\right)  .
\]
Using similar arguments to the second case, then
\[
\det\delta_{\Lambda}\left(  f\right)  =\left\vert
\begin{array}
[c]{cccc}%
1 & \omega^{\ell_{1}} & \cdots & \omega^{\left(  n-1\right)  \ell_{1}}\\
1 & \omega^{\ell_{2}} & \cdots & \omega^{\left(  n-1\right)  \ell_{2}}\\
\vdots & \vdots &  & \vdots\\
1 & \omega^{\ell_{n-1}} & \cdots & \omega^{\left(  n-1\right)  \ell_{p}}%
\end{array}
\right\vert {\displaystyle\prod\limits_{j=0}^{n-1}}f_{j}\neq0.
\]
This completes the proof.
\end{proof}

\begin{example}
Let $n=7$. Let us suppose that we pick a subset $\Lambda\ $of $\Sigma$ of
cardinality four such that $\delta_{\Lambda}\left(  \left[  f_{0},f_{1}%
,f_{2},f_{3},f_{4},f_{5},f_{6}\right]  ^{T}\right)  ^{T}$ is equal to
\[
\left(
\begin{array}
[c]{ccccccc}%
f_{0} & f_{0} & f_{0} & f_{0} & f_{0} & f_{0} & f_{0}\\
f_{1}e^{\frac{2}{7}i\pi k_{1}} & f_{1}e^{\frac{2}{7}i\pi k_{2}} &
f_{1}e^{\frac{2}{7}i\pi k_{3}} & f_{1}e^{\frac{2}{7}i\pi k_{4}} &
f_{6}e^{\frac{2}{7}i\pi\ell_{1}} & f_{6}e^{\frac{2}{7}i\pi\ell_{2}} &
f_{6}e^{\frac{2}{7}i\pi\ell_{3}}\\
f_{2}e^{\frac{4}{7}i\pi k_{1}} & f_{2}e^{\frac{4}{7}i\pi k_{2}} &
f_{2}e^{\frac{4}{7}i\pi k_{3}} & f_{2}e^{\frac{4}{7}i\pi k_{4}} &
f_{5}e^{\frac{4}{7}i\pi\ell_{1}} & f_{5}e^{\frac{4}{7}i\pi\ell_{2}} &
f_{5}e^{\frac{4}{7}i\pi\ell_{3}}\\
f_{3}e^{\frac{6}{7}i\pi k_{1}} & f_{3}e^{\frac{6}{7}i\pi k_{2}} &
f_{3}e^{\frac{6}{7}i\pi k_{3}} & f_{3}e^{\frac{6}{7}i\pi k_{4}} &
f_{4}e^{\frac{6}{7}i\pi\ell_{1}} & f_{4}e^{\frac{6}{7}i\pi\ell_{2}} &
f_{4}e^{\frac{6}{7}i\pi\ell_{3}}\\
f_{4}e^{\frac{8}{7}i\pi k_{1}} & f_{4}e^{\frac{8}{7}i\pi k_{2}} &
f_{4}e^{\frac{8}{7}i\pi k_{3}} & f_{4}e^{\frac{8}{7}i\pi k_{4}} &
f_{3}e^{\frac{8}{7}i\pi\ell_{1}} & f_{3}e^{\frac{8}{7}i\pi\ell_{2}} &
f_{3}e^{\frac{8}{7}i\pi\ell_{3}}\\
f_{5}e^{\frac{10}{7}i\pi k_{1}} & f_{5}e^{\frac{10}{7}i\pi k_{2}} &
f_{5}e^{\frac{10}{7}i\pi k_{3}} & f_{5}e^{\frac{10}{7}i\pi k_{4}} &
f_{2}e^{\frac{10}{7}i\pi\ell_{1}} & f_{2}e^{\frac{10}{7}i\pi\ell_{2}} &
f_{2}e^{\frac{10}{7}i\pi\ell_{3}}\\
f_{6}e^{\frac{12}{7}i\pi k_{1}} & f_{6}e^{\frac{12}{7}i\pi k_{2}} &
f_{6}e^{\frac{12}{7}i\pi k_{3}} & f_{6}e^{\frac{12}{7}i\pi k_{4}} &
f_{1}e^{\frac{12}{7}i\pi\ell_{1}} & f_{1}e^{\frac{12}{7}i\pi\ell_{2}} &
f_{1}e^{\frac{12}{7}i\pi\ell_{3}}%
\end{array}
\right)  .
\]
The monomial isolated in the proof of Proposition \ref{prime case} to show
that
\[
\det\left(  \delta_{\Lambda}\left(  \left[  f_{0},f_{1},f_{2},f_{3}%
,f_{4},f_{5},f_{6}\right]  ^{T}\right)  \right)
\]
is a non-trivial polynomial is: $f_{0}f_{1}^{2}f_{2}^{2}f_{3}^{2}.$ The
coefficient of $f_{0}f_{1}^{2}f_{2}^{2}f_{3}^{2}$ in the polynomial
$\det\left(  \delta_{\Lambda}\left(  \left[  f_{0},f_{1},f_{2},f_{3}%
,f_{4},f_{5},f_{6}\right]  ^{T}\right)  \right)  $ is given by
\begin{equation}
\left\vert
\begin{array}
[c]{cccc}%
1 & 1 & 1 & 1\\
e^{\frac{2}{7}i\pi k_{1}} & e^{\frac{2}{7}i\pi k_{2}} & e^{\frac{2}{7}i\pi
k_{3}} & e^{\frac{2}{7}i\pi k_{4}}\\
e^{\frac{4}{7}i\pi k_{1}} & e^{\frac{4}{7}i\pi k_{2}} & e^{\frac{4}{7}i\pi
k_{3}} & e^{\frac{4}{7}i\pi k_{4}}\\
e^{\frac{6}{7}i\pi k_{1}} & e^{\frac{6}{7}i\pi k_{2}} & e^{\frac{6}{7}i\pi
k_{3}} & e^{\frac{6}{7}i\pi k_{4}}%
\end{array}
\right\vert \left\vert
\begin{array}
[c]{ccc}%
e^{\frac{8}{7}i\pi\ell_{1}} & e^{\frac{8}{7}i\pi\ell_{2}} & e^{\frac{8}{7}%
i\pi\ell_{3}}\\
e^{\frac{10}{7}i\pi\ell_{1}} & e^{\frac{10}{7}i\pi\ell_{2}} & e^{\frac{10}%
{7}i\pi\ell_{3}}\\
e^{\frac{12}{7}i\pi\ell_{1}} & e^{\frac{12}{7}i\pi\ell_{2}} & e^{\frac{12}%
{7}i\pi\ell_{3}}%
\end{array}
\right\vert . \label{minors}%
\end{equation}
Furthermore, with some formal calculations, it is easy to see that
(\ref{minors}) is equal to%
\begin{align*}
&  -\left(  e^{\frac{2}{7}i\pi k_{1}}-e^{\frac{2}{7}i\pi k_{2}}\right)
\left(  e^{\frac{2}{7}i\pi k_{1}}-e^{\frac{2}{7}i\pi k_{3}}\right)
\allowbreak\left(  e^{\frac{2}{7}i\pi k_{1}}-e^{\frac{2}{7}i\pi k_{4}}\right)
\\
&  \left(  e^{\frac{2}{7}i\pi k_{2}}-e^{\frac{2}{7}i\pi k_{3}}\right)
\allowbreak\left(  e^{\frac{2}{7}i\pi k_{2}}-e^{\frac{2}{7}i\pi k_{4}}\right)
\left(  e^{\frac{2}{7}i\pi k_{3}}-e^{\frac{2}{7}i\pi k_{4}}\right) \\
&  \left(  \allowbreak e^{\frac{8}{7}i\pi\ell_{1}}e^{\frac{3}{2}(\frac{8}%
{7}i\pi\ell_{2})}e^{\frac{5}{4}(\frac{8}{7}i\pi\ell_{3})}-e^{\frac{8}{7}%
i\pi\ell_{1}}e^{\frac{5}{4}(\frac{8}{7}i\pi\ell_{2})}e^{\frac{3}{2}(\frac
{8}{7}i\pi\ell_{3})}\right. \\
&  -e^{\frac{3}{2}(\frac{8}{7}i\pi\ell_{1})}e^{\frac{8}{7}i\pi\ell_{2}%
}e^{\frac{5}{4}(\frac{8}{7}i\pi\ell_{3})}+e^{\frac{3}{2}(\frac{8}{7}i\pi
\ell_{1})}e^{\frac{5}{4}(\frac{8}{7}i\pi\ell_{2})}e^{\frac{8}{7}i\pi\ell_{3}%
}\\
&  \left.  +e^{\frac{5}{4}(\frac{8}{7}i\pi\ell_{1})}e^{\frac{8}{7}i\pi\ell
_{2}}e^{\frac{3}{2}(\frac{8}{7}i\pi\ell_{3})}-e^{\frac{5}{4}(\frac{8}{7}%
i\pi\ell_{1})}e^{\frac{3}{2}(\frac{8}{7}i\pi\ell_{2})}e^{\frac{8}{7}i\pi
\ell_{3}}\right)  .
\end{align*}

\end{example}

\subsection{Proof of Theorem \ref{main result}}

The proofs of Part $1$ and $2$ of Theorem \ref{main result} follow from
Corollary \ref{even case}, Proposition \ref{prime case} and Lemma
\ref{intertwines}. Finally, Part $3$ is a direct consequence of Part $2.$

\subsection{Proof of Theorem \ref{main 3}}

Let $\tau$ be a unitary irreducible representation of $D_{2n}.$ The
classification of the irreducible representations of the Dihedral group is
well-understood (see Page $36,$ \cite{Serre}). When $n$ is even, then up to
equivalence there are four one-dimensional irreducible representations. When
$n$ is odd, up to equivalence there are two one-dimensional irreducible
representations of the Dihedral group. If $\tau$ is an irreducible
representation of $D_{2n}$ which is not a character then it is well-known that
$\tau$ must be a two-dimensional representation obtained by inducing some
character of the normal subgroup generated by $r$ to $D_{2n}.$ If $\tau$ is a
character then Part $1$ holds obviously. In fact, for any nonzero vector $v\in%
\mathbb{C}
,$ the set $\tau\left(  D_{2n}\right)  v$ has the Haar property. Now, suppose
that $\tau$ is not a character. Furthermore, assume that $n$ is odd. Then
there exists $j,$ $1\leq j\leq n-1$ and a realization of the representation
$\tau$ such that $\tau=\tau_{j},$ where
\begin{equation}
\tau_{j}\left(  r\right)  =\left(
\begin{array}
[c]{cc}%
e^{\frac{2\pi ji}{n}} & 0\\
0 & e^{-\frac{2\pi ji}{n}}%
\end{array}
\right)  \text{ and }\tau_{j}\left(  s\right)  =\left(
\begin{array}
[c]{cc}%
0 & 1\\
1 & 0
\end{array}
\right)  . \label{rep}%
\end{equation}
Similarly, in the case where $n$ is even, there exists $j,$ $j\in\left\{
1,\cdots,n-1\right\}  -\left\{  \frac{n}{2}\right\}  $ such that $\tau
=\tau_{j}$ is as described in (\ref{rep}). For Part $2,$ assume that $n$ is
prime. There are three main cases to consider. Let $v=\left(
\begin{array}
[c]{cc}%
v_{1} & v_{2}%
\end{array}
\right)  ^{T}\in\mathbb{C}^{2}$. Let us suppose that $M=\tau\left(  r\right)
^{k_{1}},N=\tau\left(  r\right)  ^{k_{2}}$ such that $k_{1}\neq k_{2}$ and
$k_{1},k_{2}\in\mathbb{Z}_{n}.$ Then%
\[
\left\vert
\begin{array}
[c]{cc}%
Mv & Nv
\end{array}
\right\vert =\left\vert
\begin{array}
[c]{cc}%
v_{1}e^{2i\pi\frac{j}{n}k_{1}} & v_{1}e^{2i\pi\frac{j}{n}k_{2}}\\
v_{2}e^{-2i\pi\frac{j}{n}k_{1}} & v_{2}e^{-2i\pi\frac{j}{n}k_{2}}%
\end{array}
\right\vert =2iv_{1}v_{2}\sin\left(  \frac{2\pi j\left(  k_{1}-k_{2}\right)
}{n}\right)  .
\]
Next, let us suppose that $M=\tau\left(  r\right)  ^{k_{1}},N=\tau\left(
r^{k_{2}}s\right)  $ where $k_{1},k_{2}\in%
\mathbb{Z}
_{n}.$ Then
\begin{align*}
\left\vert
\begin{array}
[c]{cc}%
Mv & Nv
\end{array}
\right\vert  &  =\left\vert
\begin{array}
[c]{cc}%
v_{1}e^{2i\pi\frac{j}{n}k_{1}} & v_{2}e^{2i\pi\frac{j}{n}k_{2}}\\
v_{2}e^{-2i\pi\frac{j}{n}k_{1}} & v_{1}e^{-2i\pi\frac{j}{n}k_{2}}%
\end{array}
\right\vert \\
&  =\left(  v_{1}^{2}-v_{2}^{2}\right)  \cos\left(  \frac{2\pi j\left(
k_{1}-k_{2}\right)  }{n}\right)  +i\left(  v_{1}^{2}+v_{2}^{2}\right)
\sin\left(  \frac{2\pi j\left(  k_{1}-k_{2}\right)  }{n}\right)  .
\end{align*}
Finally, let us suppose that
\[
M=\tau\left(  r^{k_{1}}s\right)  ,N=\tau\left(  r^{k_{2}}s\right)
\]
such that $k_{1}\neq k_{2}$ and $k_{1},k_{2}\in%
\mathbb{Z}
_{n}.$ Then
\[
\left\vert
\begin{array}
[c]{cc}%
Mv & Nv
\end{array}
\right\vert =\left\vert
\begin{array}
[c]{cc}%
v_{2}e^{2i\pi\frac{j}{n}k_{1}} & v_{2}e^{2i\pi\frac{j}{n}k_{2}}\\
v_{1}e^{-2i\pi\frac{j}{n}k_{1}} & v_{1}e^{-2i\pi\frac{j}{n}k_{2}}%
\end{array}
\right\vert =2iv_{1}v_{2}\sin\left(  \frac{2\pi j\left(  k_{1}-k_{2}\right)
}{n}\right)  .
\]
Next, it is easy to check that the polynomials
\[
p\left(  v_{1},v_{2}\right)  =2iv_{1}v_{2}\sin\left(  \frac{2\pi j\left(
k_{1}-k_{2}\right)  }{n}\right)  ,
\]
and%
\begin{equation}
p^{\prime}\left(  v_{1},v_{2}\right)  =\left(  v_{1}^{2}-v_{2}^{2}\right)
\cos\left(  \frac{2\pi j\left(  k_{1}-k_{2}\right)  }{n}\right)  +i\left(
v_{1}^{2}+v_{2}^{2}\right)  \sin\left(  \frac{2\pi j\left(  k_{1}%
-k_{2}\right)  }{n}\right)  \label{po}%
\end{equation}
are all non-trivial whenever $n$ is prime. Indeed, if $n$ is prime,
$k_{1}-k_{2}\in\left\{  1,\cdots,n-1\right\}  $, $j\in\left\{  1,\cdots
,n-1\right\}  ,$ then the real number $\frac{2\pi j\left(  k_{1}-k_{2}\right)
}{n}$ can never be equal to $\pi\ell$ where $\ell\in\mathbb{Z}.$ So, $p\left(
v_{1},v_{2}\right)  $ is a nonzero homogeneous polynomial in $v_1, v_2.$ Next, since the
coefficient of the monomial $v_{1}^{2}$ in (\ref{po}) is given by
$e^{i\frac{2\pi j\left(  k_{1}-k_{2}\right)  }{n}}$ then $p^{\prime}\left(
v_{1},v_{2}\right)  $ is a nonzero homogeneous polynomial as well. So when $n$
is prime, for any distinct matrices $M,N\in\tau\left(  D_{2n}\right)  ,$ the
set $\left\{  Mv,Nv\right\}  $ is linearly independent for almost every
$v\in\mathbb{C}^{2}.$ For Part $4$, let us assume that $n$ is an odd composite
number. Then there exist odd natural numbers $n_{1},n_{2}\in\mathbb{N}$ such
that $n=n_{1}n_{2}$ and $n_{k}\notin\left\{  1,n\right\}  .$ Next, we observe
that
\[
\left(
\begin{array}
[c]{cc}%
e^{\frac{2\pi n_{1}i}{n_{1}n_{2}}\times n_{2}} & 0\\
0 & e^{-\frac{2\pi n_{1}i}{n_{1}n_{2}}\times n_{2}}%
\end{array}
\right)  -\left(
\begin{array}
[c]{cc}%
1 & 0\\
0 & 1
\end{array}
\right)  =\left(
\begin{array}
[c]{cc}%
0 & 0\\
0 & 0
\end{array}
\right)  .
\]
So for any vector $v\in%
\mathbb{C}
^{2},$ the set $\left\{  v,\tau_{n_{1}}\left(  r^{n_{2}}\right)  v\right\}  $
is linearly dependent. Now, let us consider the representation of the Dihedral
group $\tau_{1}$ defined such that
\[
\tau_{1}\left(  r\right)  =\left(
\begin{array}
[c]{cc}%
e^{\frac{2\pi i}{n}} & 0\\
0 & e^{-\frac{2\pi i}{n}}%
\end{array}
\right)  \text{ and }\tau_{1}\left(  s\right)  =\left(
\begin{array}
[c]{cc}%
0 & 1\\
1 & 0
\end{array}
\right)  .
\]
For distinct matrices $M,N\in\tau\left(  D_{2n}\right)  $, either
\[
\left\vert
\begin{array}
[c]{cc}%
Mv & Nv
\end{array}
\right\vert =2iv_{1}v_{2}\sin\left(  \frac{2\pi\left(  k_{1}-k_{2}\right)
}{n}\right)  ,\text{ }k_{1}-k_{2}\in\left\{  1,\cdots,m-1\right\}
\]
or
\[
\left\vert
\begin{array}
[c]{cc}%
Mv & Nv
\end{array}
\right\vert =\left(  v_{1}^{2}-v_{2}^{2}\right)  \cos\left(  \frac{2\pi\left(
k_{1}-k_{2}\right)  }{n}\right)  +i\left(  v_{1}^{2}+v_{2}^{2}\right)
\sin\left(  \frac{2\pi\left(  k_{1}-k_{2}\right)  }{n}\right)  ,\text{ }%
\]
where $k_{1},k_{2}\in\left\{  1,\cdots,n-1\right\}  .$ Since $n$ is assumed to
be odd and because $k_{1}-k_{2}\in\left\{  1,\cdots,n-1\right\}  ;$ it is easy
to see that $2iv_{1}v_{2}\sin\left(  \frac{2\pi\left(  k_{1}-k_{2}\right)
}{n}\right)  $ is a non-trivial homogeneous polynomial. To show this, let us
suppose that for $k=k_{1}-k_{2},$ $\frac{2\pi k}{n}=\pi\ell\text{ for some
}\ell\in\mathbb{Z}.$ Then $k=\frac{1}{2}n\ell.$ Since $n$ is odd then
$\ell=2\ell^{\prime}$ for some $\ell^{\prime}\in\mathbb{Z}.$ It follows that
$k$ is a multiple of $n,$ and this is impossible. Next, the fact that
\[
\left(  v_{1}^{2}-v_{2}^{2}\right)  \cos\left(  \frac{2\pi\left(  k_{1}%
-k_{2}\right)  }{n}\right)  +i\left(  v_{1}^{2}+v_{2}^{2}\right)  \sin\left(
\frac{2\pi\left(  k_{1}-k_{2}\right)  }{n}\right)
\]
is a nonzero polynomial was already proved for Part $2$. Finally for Part $3,$
let us assume that $n=2j$ is even, and $j\in\left\{  1,\cdots,n-1\right\}
-\left\{  \frac{n}{2}\right\}  $. If $j$ is odd, then it is easy to see that
\[
\left(
\begin{array}
[c]{cc}%
1 & 0\\
0 & 1
\end{array}
\right)  +\left(
\begin{array}
[c]{cc}%
e^{\frac{2\pi ji}{n}\frac{n}{2}} & 0\\
0 & e^{-\frac{2\pi ji}{n}\frac{n}{2}}%
\end{array}
\right)  =\left(
\begin{array}
[c]{cc}%
0 & 0\\
0 & 0
\end{array}
\right)  .
\]
Also, if $j$ is even then
\[
-\left(
\begin{array}
[c]{cc}%
1 & 0\\
0 & 1
\end{array}
\right)  +\left(
\begin{array}
[c]{cc}%
e^{\frac{2\pi ji}{n}\frac{n}{2}} & 0\\
0 & e^{-\frac{2\pi ji}{n}\frac{n}{2}}%
\end{array}
\right)  =\left(
\begin{array}
[c]{cc}%
0 & 0\\
0 & 0
\end{array}
\right)  .
\]
Thus it is not possible to find a vector in $%
\mathbb{C}
^{2}$ such that for any distinct matrices $M,N\in\tau\left(  D_{2n}\right)  ,$
the set $\left\{  Mv,Nv\right\}  $ is linearly independent. This completes the proof.

\section{Examples}

\begin{example}
Let $n=3.$ For any subset $\Lambda$ of $\Sigma$ of cardinality $3,$ it is not
too hard to see that the polynomial $\delta_{\Lambda}\left(  \left(
1,z,z^{4}\right)  ^{T}\right)  $ is a nonzero polynomial of degree at most $8$
in the variable $z.$ Thus, given any algebraic number $x$ of degree at least
$9$ over the cyclotomic field $\mathbb{Q}\left(  e^{\frac{2\pi i}{3}}\right)
$ or given any transcendental number $x,$ $x$ cannot be a root of the
polynomial $\delta_{\Lambda}\left(  \left(  1,z,z^{4}\right)  ^{T}\right)  $.
It follows that the set $\Gamma\mathbf{F}^{-1}\left(  1,z,z^{4}\right)  ^{T}$
is a frame in $\mathbb{C}^{3}$ which is maximally robust to erasures. For
example, let
\[
v=\left(
\begin{array}
[c]{ccc}%
\frac{1}{3}\sqrt{3} & \frac{1}{3}\sqrt{3} & \frac{1}{3}\sqrt{3}\\
\frac{1}{3}\sqrt{3} & \frac{1}{2}i-\frac{1}{6}\sqrt{3} & -\frac{1}{6}\sqrt
{3}-\frac{1}{2}i\\
\frac{1}{3}\sqrt{3} & -\frac{1}{6}\sqrt{3}-\frac{1}{2}i & \frac{1}{2}%
i-\frac{1}{6}\sqrt{3}%
\end{array}
\right)  ^{-1}\left(
\begin{array}
[c]{c}%
1\\
\pi\\
\pi^{2}%
\end{array}
\right)  =\left(
\begin{array}
[c]{c}%
\sqrt{3}\left(  \frac{1}{3}\pi+\frac{1}{3}\pi^{2}+\frac{1}{3}\right) \\
-\frac{1}{6}\left(  \pi-1\right)  \left(  \sqrt{3}\pi-3i\pi+2\sqrt{3}\right)
\\
-\frac{1}{6}\left(  \pi-1\right)  \left(  3i\pi+\sqrt{3}\pi+2\sqrt{3}\right)
\end{array}
\right)  .
\]
Then
\[
\Gamma\left(
\begin{array}
[c]{c}%
\sqrt{3}\left(  \frac{1}{3}\pi+\frac{1}{3}\pi^{2}+\frac{1}{3}\right) \\
-\frac{1}{6}\left(  \pi-1\right)  \left(  \sqrt{3}\pi-3i\pi+2\sqrt{3}\right)
\\
-\frac{1}{6}\left(  \pi-1\right)  \left(  3i\pi+\sqrt{3}\pi+2\sqrt{3}\right)
\end{array}
\right)
\]
is a frame in $\mathbb{C}^{3}$ which is maximally robust to erasures.
\end{example}

\begin{example}
Let $n=5.$ Put
\[
v=\left(
\begin{array}
[c]{ccccc}%
i & -i & 1 & 1+i & 2-i
\end{array}
\right)  ^{T}.
\]
Using Mathematica, we are able to show that $\Gamma v$ is a frame in
$\mathbb{C}^{5}$ which is maximally robust to erasures. Put
\[
M_{\Gamma}\left(  v\right)  =\left(
\begin{array}
[c]{ccccc}%
i & -i & 1 & 1+i & 2-i\\
2-i & i & -i & 1 & 1+i\\
1+i & 2-i & i & -i & 1\\
1 & 1+i & 2-i & i & -i\\
-i & 1 & 1+i & 2-i & i\\
i & 2-i & 1+i & 1 & -i\\
-i & i & 2-i & 1+i & 1\\
1 & -i & i & 2-i & 1+i\\
1+i & 1 & -i & i & 2-i\\
2-i & 1+i & 1 & -i & i
\end{array}
\right)  .
\]
Each row of the matrix above corresponds to an element of the orbit of $v.$
Thus every sub-matrix of $M_{\Gamma}\left(  v\right)  $ of order five is invertible.
\end{example}

\end{document}